\documentclass[12pt, twoside]{article}
\usepackage{times}
\usepackage{bm}
\usepackage{graphics}
\usepackage{theorem}
\usepackage{graphicx}
\usepackage{amsmath}
\usepackage{latexsym}
\usepackage{amssymb,mathrsfs}
\usepackage{cases}
\usepackage[flushmargin]{footmisc}

\makeatletter
\def\eqnarray{\stepcounter{equation}\let\@currentlabel=\theequation
\global\@eqnswtrue
\global\@eqcnt\z@\tabskip\@centering\let\\=\@eqncr
$$\halign to \displaywidth\bgroup\@eqnsel\hskip\@centering
  $\displaystyle\tabskip\z@{##}$&\global\@eqcnt\@ne 
  \hfil$\;{##}\;$\hfil
  &\global\@eqcnt\tw@ $\displaystyle\tabskip\z@{##}$\hfil 
   \tabskip\@centering&\llap{##}\tabskip\z@\cr}
\makeatother
\theorembodyfont{\itshape}
\newtheorem{thm}{Theorem}[section]
\newtheorem{lem}{Lemma}[section]

{\bf}{\rm}

\theorembodyfont{\rmfamily}
{\bf}{\rm}
{\bf}{\rm}

\newtheorem{rem}{Remark}[section]{\itshape}{\rmfamily}

\newenvironment{proof}{\noindent{\it Proof.~~}}{\qed\medskip}

\makeatletter
    \renewcommand{\theequation}{%
    \thesection.\arabic{equation}}
    \@addtoreset{equation}{section}
  \makeatother

\newcommand{\vc}{\bm}
\newcommand{\down}[2]{\smash{\lower#2\hbox{#1}}}
\newcommand{\up}[2]{\smash{\lower-#2\hbox{#1}}}

\newcommand{\qed}{\hspace*{\fill}$\Box$\rule[-10pt]{0pt}{10pt}}

\newcommand{\PP}{\mathsf{P}}

\newcommand{\bbD}{\mathbb{D}}

\newcommand{\bbF}{\mathbb{F}}

\newcommand{\bbL}{\mathbb{L}}
\newcommand{\bbK}{\mathbb{K}}

\newcommand{\bbN}{\mathbb{N}}

\newcommand{\bbZ}{\mathbb{Z}}

\newcommand{\rmd}{{\rm d}}
\newcommand{\rme}{{\rm e}}

\newcommand{\ol}{\overline}
\newcommand{\ul}{\underline}

\newcommand{\wh}{\widehat}

\newcommand{\dd}[1]{\if#11 1\!\!1 
\else {\if#1C I\!\!\!C
\else {\if#1G I\!\!\!G 
\else {\if#1J J\!\!\!J 
\else {\if#1S S\!\!\!S
\else {\if#1Z Z\!\!\!Z
\else {\if#1Q O\!\!\!\!Q
\else I\!\!#1
\fi} 
\fi}
\fi}
\fi} 
\fi} 
\fi} 
\fi} 

\makeatother

\begin{document}\thispagestyle{plain} 

\hfill
{\small Last update date: \today}

{\Large{\bf
\begin{center}
The stability condition of BMAP/M/{\LARGE$\infty$} queues
\end{center}
}
}

\begin{center}
{
Moeko Yajima\footnote[1]{Email:~yajima.m.ad@m.titech.ac.jp}, ~~{\it \footnotesize Tokyo Institute of Technology}

Tuan Phung-Duc\footnote[2]{Email:~tuan@sk.tsukuba.ac.jp},~~{\it \footnotesize University of Tsukuba} 

Hiroyuki Masuyama\footnote[3]{Email:~masuyama@sys.i.kyoto-u.ac.jp},~~{\it \footnotesize Kyoto University}
}

\bigskip
\medskip

{\small
\textbf{Abstract}

\medskip

\begin{tabular}{p{0.85\textwidth}}
This paper considers a BMAP/M/$\infty$ queue with a batch Markovian
arrival process (BMAP) and an exponential service time
distribution. We first prove that the BMAP/M/$\infty$ queue is stable
if and only if the expectation of the logarithm of the batch-size
distribution is finite. Using this result, we also present the
stability condition for an infinite-server queue with a multiclass
batch Markovian arrival process and class-dependent exponential
service times.
\end{tabular}
}
\end{center}

\begin{center}
\begin{tabular}{p{0.90\textwidth}}
{\small


%
}
\end{tabular}

\end{center}

\section{Introduction}

Infinite-server queues have many applications in various areas, such
as inventory systems \cite{Berm99}, road traffic systems \cite{Woen07}
and telecommunication systems \cite{Mass02}. Thus, many researchers have studied stationary and/or time-dependent infinite-server queues (see, e.g., \cite{Chatterjee89,Holm83,Keilson88,Rama80,Shanbhag66,Taka80} and the references therein). However, almost all the previous works paid little
attention to the {\it stability condition} of infinite-server queues,
that is, the necessary and sufficient condition that there exists the
unique stationary distribution of the queue length process (i.e., the
stochastic process of the number of busy servers). This would be
because infinite-server queues with individual arrivals are always
stable.  On the other hand, infinite-server queues with batch arrivals
are not always stable (see e.g., Cong~\cite{Cong94}).

As far as we know, all the previous works, except Cong~\cite{Cong94},
have studied stationary infinite-server queues with batch arrivals,
assuming sufficient conditions for stability (e.g., the first two
moments of the batch-size distribution are finite) or the existence of
the stationary queue length distribution.

Holman \textit{et al.} \cite{Holm83} derived some formulas for the
mean and variance of the stationary queue length distribution in the
M${}^X$/G/$\infty$ queue, under the assumption that the first two moments of the batch-size distribution are finite.  Keilson and
Seidmann~\cite{Keilson88} assumed that the M${}^X$/G/$\infty$ queue is
stable and then proved that the stationary queue length distribution
is a compound Poisson distribution under an additional condition.
Breuer{~}\cite{Breu12} derived the necessary and sufficient condition
that the mean stationary queue length in the
BMAP/G/$\infty$ queue is finite.

As for the multiclass case, Liu and Templeton{~}\cite{LiuLim91}
considered an infinite-server queue (referred to as the
GR${}^{X_n}$/G${}_n$/$\infty$ queue therein), where the arrival times
and types of customers are governed by a Markov renewal process and
the batch sizes of customers depend on their types. For the
GR${}^{X_n}$/G${}_n$/$\infty$ queue, Liu and
Templeton{~}\cite{LiuLim91} derived the probability generating
function of the stationary queue length distribution under the
assumption that all the moments of the batch-size distribution are
finite.  Masuyama and Takine{~}\cite{Masu02} derived explicit and
numerically feasible formulas for the stationary joint queue length
moments in an infinite-server queue with a multiclass batch Markovian
arrival process and class-dependent phase-type service times, assuming
that the stationary joint queue length distribution exists.

Unlike the previous works mentioned above, Cong~\cite{Cong94} paid an
attention to the stability condition for infinite-server queues with
batch arrivals.  In fact, Cong~\cite{Cong94} established the stability
condition of the multiclass M${}^X$/M/$\infty$ queue, where customers
arrive according to a multiclass batch Poisson process and
class-dependent exponential service times.  Cong~\cite{Cong94}'s
stability condition is that the first logarithmic moment of the
batch-size distribution is finite, i.e., the mean value of the
logarithm of the batch size is finite. For convenience, we refer to
the stability condition of this type as the {\it logarithmic moment condition}.

The main purpose of this paper is to prove that the logarithmic moment
condition is the stability condition of the BMAP/M/$\infty$ queue,
which includes the M${}^X$/M/$\infty$ queue as a special case. Using
Foster's theorem (see, e.g., \cite[Chapter~5, Theorem 1.1]{Brem99}),
we prove that the logarithmic moment condition is sufficient for the
stability of the BMAP/M/$\infty$ queue. We also show the necessity of
the logarithmic moment condition for stability in a similar way to
Cong~\cite{Cong94}. In addition, combining these results with the
stochastic ordering technique, we prove that the logarithmic moment
condition is the stability condition of a multiclass BMAP/M/$\infty$
queue, where customers arrive according to a multiclass batch
Markovian arrival process (MBMAP) and service times of customers are independently distributed with class-dependent exponential distributions.

The reminder of this paper is organized as
follows. Section~\ref{sec-Model} describes the BMAP/M/$\infty$
queue. Sections~\ref{Stability Condition} and \ref{sec-Multi} discuss
the stability condition for the BMAP/M/$\infty$ queue and the
multiclass BMAP/M/$\infty$ queue. Finally, Section~\ref{Contribution}
is devoted to concluding remarks and future work.
\section{Model Description}\label{sec-Model}

In this section, we describe the BMAP/M/$\infty$ queue. This queueing model has infinite servers, where customers arrive according
to a batch Markovian arrival process (BMAP) \cite{Luca91}.  The BMAP
includes various arrival processes as special cases, e.g., a batch
Poisson arrival process, a Phase-type (PH) renewal process
\cite{Lato99}, a Markovian arrival process (MAP) \cite{Luca90}. Note
here that the MAP is an special case of BMAPs such that arrivals occur
one by one.  It is known \cite{Asmu93} that any simple point process
is the weak limit of a sequence of MAPs.

The BMAP is controlled by an irreducible time-homogeneous Markov chain
$\{J(t);t\geq 0\}$ in continuous time with finite state space
$\bbD:=\{1,2,\dots,d\}$, which is called the background Markov chain.
Let $N(t)$, $t \ge 0$, denote the total number of customers arriving
from the BMAP during the time interval $(0,t]$, where $N(0)=0$. We
  assume that, for $k\in\bbZ_+:=\{0,1,\dots\}$ and $i,j\in\bbD$,
\begin{eqnarray*}
\lefteqn{
\PP(N(t+\Delta t)-N(t)=k,J(t+\Delta t)=j\ |\ J(t)=i)
}\quad&&
\\
&=&\left\{\begin{array}{ll}
1+D_{i,i}(0)\Delta t + o(\Delta t),&k=0,\ i=j\in\bbD,\\
D_{i,j}(k)\Delta t + o(\Delta t),& \mbox{otherwise},
\end{array}\right.
\end{eqnarray*}
where $f(x)=o(g(x))$ represents $\lim_{x\downarrow 0}|f(x)|/|g(x)|=0$.
Note here that $\vc{D}(k):=(D_{i,j}(k))_{i,j\in\bbD}$,
$k\in\bbN:=\{1,2,\dots\}$, is a nonnegative matrix and that $D(0):=(D_{i,j}(0))_{i,j\in\bbD}$ is a diagonally dominant matrix with negative diagonal and nonnegative off-diagonal elements because of the irreducibility of the background Markov chain ${J(t)}$. Note also that $\vc{D}:=\sum_{k=0}^{\infty} \vc{D}(k)$ is
the infinitesimal generator of the background Markov chain
$\{J(t)\}$. To avoid triviality, we assume that
\begin{equation}
\sum_{k=1}^{\infty} \vc{D}(k)\vc{e}\neq\vc{0},
\label{sum-D(n)e}
\end{equation}
where $\vc{e}$ and $\vc{0}$ are the column vectors of $1$'s and $0$'s,
respectively.

It is obvious that the joint stochastic process $\{(N(t),J(t));t\geq
0\}$ is a Markov chain with state space $\bbZ_+ \times \bbD$, whose
infinitesimal generator is given by
\[
\bordermatrix{
        & 
\bbL(0) &
\bbL(1) &
\bbL(2) &
\bbL(3) &
\cdots
\cr
\bbL(0) & 
\vc{D}(0) &
\vc{D}(1) &
\vc{D}(2) &
\vc{D}(3) &
\cdots
\cr
\bbL(1)   & 
\vc{O}    &
\vc{D}(0) &
\vc{D}(1) &
\vc{D}(2) &
\cdots
\cr
\bbL(2)   & 
\vc{O}    & 
\vc{O}    &
\vc{D}(0) &
\vc{D}(1) &
\cdots
\cr
\vspace{-1mm}\bbL(3)   & 
\vc{O}    & 
\vc{O}    & 
\vc{O}    &
\vc{D}(0) &
\cdots
\cr
~~~\vdots    &
\vdots    &
\vdots    &
\vdots    &
\vdots    &
\ddots    &
},
\]
where $\vc{O}$ denotes the zero matrix and $\bbL(k) = \{k\} \times
\bbD$ for $k \in \bbZ_+$. As a result, the BMAP is characterized by
$\{\vc{D}(k);k\in\bbZ_+\}$ and thus is referred to as BMAP $\{\vc{D}(k);k\in\bbZ_+\}$.

Each arriving customer occupies one of the servers immediately after
its arrival, and leaves the system immediately after its service
completion.  The service times of customers are independently and identically distributed (i.i.d.) with the exponential distribution having mean $1/\mu\in (0,\infty)$. Therefore,
customers behave independently of each other once they enter the
system.

Let $L(t)$, $t \ge 0$, denote the number of customers in the system at
time $t$.  It then follows from the Markov property of the BMAP and
exponential service times that the stochastic process
$\{(L(t),J(t)); t\geq 0\}$ is a continuous-time Markov chain with
state space $\bbF:=\bbZ_+\times\bbD$. Let
$\vc{Q}:=(q(k,i;\ell,j))_{(k,i),(\ell,j)\in\bbF}$ denote the
infinitesimal generator of the Markov chain $\{(L(t),J(t))\}$. We then
have
\begin{eqnarray}
\vc{Q}=
\left(
\begin{array}{ccccc}
\vc{D}(0) &
\vc{D}(1) &
\vc{D}(2) &
\vc{D}(3) &
\cdots
\\
\mu\vc{I} 			&
\vc{\varLambda}_1(0)&
\vc{D}(1)			&
\vc{D}(2)			&
\cdots
\\
\vc{O}		&
2\mu\vc{I}	&
\vc{\varLambda}_2(0) &
\vc{D}(1)	&
\cdots
\\
\vc{O} &
\vc{O} &
3\mu\vc{I}&
\vc{\varLambda}_3(0)&
\cdots
\\
\vspace{-1mm}\vc{O}	&
\vc{O}	&
\vc{O}	&
4\mu\vc{I}&
\up{$\ddots$}{-0.6mm}
\\
\vdots	&
\vdots	&
\vdots	&
\ddots	&
\ddots
\end{array}
\right),
\label{defn-Q}
\end{eqnarray}
where $\vc{\varLambda}_k(0) = -k\mu\vc{I}+\vc{D}(0)$ for $k \in \bbN$
and $\vc{I}$ denotes the identity matrix.

\begin{rem}\label{rem-Q}
It follows from (\ref{sum-D(n)e}), (\ref{defn-Q}) and the
irreducibility of the background Markov chain $\{J(t)\}$ that $\vc{Q}$
is irreducible. Therefore, if there exists a stationary distribution
of the Markov chain $\{(L(t),J(t))\}$ (i.e., a
stationary probability vector of $\vc{Q}$), then it is unique and
positive (see, e.g., \cite[Chapter 8, Theorem~5.1]{Brem99}).
\end{rem}

\section{Stability Condition}\label{Stability Condition}

The main purpose of this section is to present a sufficient and necessary condition for the stability (i.e.,
ergodicity) of the BMAP/M/$\infty$ queue, described in the previous
section. 

The following theorem is the main result of this paper.
\begin{thm}\label{cond-batch}
The Markov chain $\{(L(t),J(t)); t\geq 0\}$ is ergodic
if and only if there exists some finite constant $C > 0$ such that
\begin{eqnarray}
\label{stability-cond}
\sum_{k=1}^{\infty}\log(k+\rme)\vc{D}(k)\vc{e} \le C\vc{e},
\end{eqnarray}
where $\rme$ is the Napier's constant.
\end{thm}

It follows from (\ref{stability-cond}) that the time average of the
logarithm of the number of customers arriving in a bacth is finite. 
Thus, Theorem~\ref{cond-batch} shows that the logarithmic moment condition (\ref{stability-cond}) is a sufficient and necessary condition for the stability of the BMAP/M/$\infty$ queue.

In the rest of this section, we separately prove the necessity and
sufficiency of the logarithmic moment condition (\ref{stability-cond}) for the stability of the BMAP/M/$\infty$ queue.
\subsection{Sufficient Condition}
\label{sec-Sufficient}

We begin with the following lemma.
\begin{lem}\label{lem-sufficiency}
For $(k,i),(\ell,j)\in\bbF$, let $\upsilon(k,i)$ and $1_K(\ell,j)$ denote
\begin{eqnarray*}
\upsilon(k,i)
&=& \log(k+\rme), \quad k\in\bbZ_+,\ i \in \bbD,
\\
1_K(\ell,j)
&=& \left\{\begin{array}{ll@{~}l}
\displaystyle 1, & \ell = 0,1,\dots,K,  & j\in\bbD,
\\
\displaystyle 0, & \ell = K+1,K+2,\dots,& j\in\bbD,
\end{array}\right.
\end{eqnarray*}
respectively. If (\ref{stability-cond}) holds, then there exist some
$\delta \in (0,\infty)$ and $K\in\bbZ_+$ such that
\begin{eqnarray}
\vc{Q}\vc{v} \le -\delta\vc{e} + (\delta + C)\vc{1}_K,
\label{lem1}
\end{eqnarray}
where $\vc{v} = (\upsilon(k,i))_{(k,i)\in\bbF}$ and
$\vc{1}_K=(1_K(\ell,j))_{({\ell},i)\in\bbF}$.
\end{lem}

\medskip
It is immediate from Lemma \ref{lem-sufficiency} and Foster's theorem
(see, e.g., \cite[Chapter 2, Statement 8]{Fali97}) that
(\ref{stability-cond}) is a sufficient condition for the ergodicity of the
irreducible generator $\vc{Q}$.

\medskip

{\sc Proof of Lemma~\ref{lem-sufficiency}.~\,} We define $\vc{y}(k)$, $k\in\bbZ_+$, as
\[
\vc{y}(k)
=\sum_{\ell=0}^{\infty}\vc{Q}(k;\ell)\vc{v}(\ell),
\quad k\in\bbZ_+,
\]
where $\vc{Q}(k;\ell)=(q(k,i;\ell,j))_{i,j\in\bbD}$ for
$k,\ell\in\bbZ_+$ and
\[
\vc{v}(k) =(\upsilon(k,i))_{i\in\bbD}=\log(k+\rme)\vc{e},
\quad
k\in\bbZ_+.
\]
We then have
\begin{eqnarray}
\label{y0}
\vc{y}(0) 
&=& \sum_{\ell=0}^{\infty}
\log(\ell+\rme)\vc{D}(\ell)\vc{e}
\nonumber
\\
&=& \vc{D}(0)\vc{e} 
+ \sum_{\ell=1}^{\infty} \log(\ell+\rme) \vc{D}(\ell)\vc{e}
\le C\vc{e},
\end{eqnarray}
where the last inequality follows from (\ref{stability-cond}) and
$\vc{D}(0)\vc{e} \le \vc{0}$.  We also have, for $k \in \bbN$,
\begin{eqnarray}
\vc{y}(k)
&=&k\mu[\vc{v}(k-1)-\vc{v}(k)]
+\sum_{\ell=0}^{\infty}\vc{D}(\ell)\vc{v}(\ell+k)
\nonumber
\\
&=& k\mu\log\Big(1-{1 \over k+\rme}\Big)\vc{e}
+ \sum_{\ell=0}^{\infty}\vc{D}(\ell)\log(\ell+k+\rme)\vc{e}.
\quad~
\label{yk-1}
\end{eqnarray}
Note here that
\[
\log(\ell+k+\rme)
= \log(k+\rme) + \log\Big(1+{\ell \over k+\rme}\Big),
\quad k,\ell \in \bbZ_+.
\]
Using this equation and $\vc{D}\vc{e}=\vc{0}$, we obtain
\begin{eqnarray*}
\lefteqn{
\sum_{\ell=0}^{\infty} \log(\ell+k+\rme) \vc{D}(\ell)\vc{e}
}\quad&&
\\
&=&\log(k+\rme)\sum_{\ell=0}^{\infty}\vc{D}(\ell)\vc{e}
+ \sum_{\ell=0}^{\infty}
\log\Big(1+{\ell \over k+\rme}\Big)\vc{D}(\ell)\vc{e}
\\
&=&\sum_{\ell=1}^{\infty}
\log\Big(1+{\ell \over k+\rme}\Big)\vc{D}(\ell)\vc{e},
\qquad k \in \bbN.
\end{eqnarray*}
It follows from this equation and (\ref{yk-1}) that
\begin{eqnarray}
\vc{y}(k)
&=& k\mu\log\Big(1-{1 \over k+\rme}\Big)\vc{e}
\nonumber
\\
&& {} 
+\sum_{\ell=1}^{\infty}\log\Big(1+{\ell \over k+\rme}\Big)\vc{D}(\ell)\vc{e},
\quad k \in \bbN.
\label{yk-2}
\end{eqnarray}

We estimate the two terms in the right hand side of (\ref{yk-2}). It
is easy to see that
\[
\lim_{k\to\infty}k\log\Big(1-{1 \over k+\rme}\Big)
= - 1,
\]
which shows that there exists some $\delta > 0$ such that
\begin{eqnarray}
k\mu\log\Big(1-{1 \over k+\rme}\Big)
\le -2\delta \quad \mbox{for all $k \in \bbN$}.
\label{add-hm-160731-01}
\end{eqnarray}
It also follows from (\ref{stability-cond}) that, for all
$k\in\bbN$,
\begin{eqnarray}
\lefteqn{
\sum_{\ell=1}^{\infty}\log\Big(1+{\ell \over k+\rme}\Big)\vc{D}(\ell)\vc{e}
}
\quad &&
\nonumber
\\
&\leq&
\sum_{\ell=1}^{\infty}\log(\ell+\rme)\vc{D}(\ell)\vc{e}
\le C\vc{e}.
\label{add-hm-160731-02}
\end{eqnarray}
Applying (\ref{add-hm-160731-01}) and (\ref{add-hm-160731-02}) to
(\ref{yk-2}), we obtain
\begin{eqnarray}
\vc{y}(k)
&\le& -2\delta \vc{e} + C\vc{e},
\qquad k \in \bbN.
\label{add-hm-160731-03}
\end{eqnarray}
In addition, by dominated convergence theorem, we have
\[
\lim_{k\to\infty}
\sum_{\ell=1}^\infty\log\Big(1+{\ell \over k+\rme}\Big)\vc{D}(\ell)\vc{e}
=\vc{0},
\]
and thus there exists some $K := K_{\delta} \in \bbZ_+$ such that, for
all $k=K+1,K+2,\dots$,
\[
\sum_{\ell=1}^\infty\log\Big(1+{\ell \over k+\rme}\Big)\vc{D}(\ell)\vc{e}
\le \delta \vc{e}.
\]
Combining this inequality, (\ref{yk-2}) and (\ref{add-hm-160731-01}),
we obtain
\begin{eqnarray}
\vc{y}(k)
&\le& -\delta \vc{e}, \qquad k=K+1,K+2,\dots.
\label{add-hm-160731-04}
\end{eqnarray}
Consequently, (\ref{lem1}) follows from (\ref{y0}),
(\ref{add-hm-160731-03}) and (\ref{add-hm-160731-04}).  \qed

\subsection{Necessary Condition}
\label{sec-Necessary}

The following lemma shows that the logarithmic moment condition (\ref{stability-cond}) holds if
$\vc{Q}$ is ergodic, i.e., $\vc{Q}$ has the unique stationary
probability vector.
\begin{lem}
If $\vc{Q}$ has the unique stationary probability vector
$\vc{\pi}=(\pi(k,i))_{(k,i)\in\bbF}$, then (\ref{stability-cond})
holds for some finite constant $C> 0$.
\end{lem}

\begin{proof}
Let $\vc{\pi}(k) = (\pi(k,i))_{i\in\bbD}$ for $k \in
\bbZ_+$, which is positive (see Remark~\ref{rem-Q}). It follows from
the global balance equation $\vc{\pi}\vc{Q} = \vc{0}$ that
\[
k \mu\vc{\pi}(k)
=(k + 1)\mu \vc{\pi}(k + 1)
+ \sum_{\ell=0}^k \vc{\pi}(k - \ell)\vc{D}(\ell),\quad k \in \bbZ_+.
\]
Multiplying the above equation by $z^k$ and taking the sum over
$k\in\bbZ_+$, we obtain, for $|z|\leq 1$,
\begin{eqnarray*}
\mu\sum_{k=1}^{\infty} k z^k \vc{\pi}(k)
&=& \mu\sum_{k=0}^{\infty} (k+1) z^k \vc{\pi}(k+1)
\nonumber
\\
&& {} +
\sum_{k=0}^{\infty}\sum_{\ell=0}^k z^k\vc{\pi}(k-\ell)\vc{D}(\ell),
\end{eqnarray*}
which leads to
\begin{equation}
\mu z{\rmd \over \rmd z}\wh{\vc{\pi}}(z)
= \mu{\rmd \over \rmd z}\wh{\vc{\pi}}(z)
+ \wh{\vc{\pi}}(z)\sum_{k=0}^{\infty}z^k \vc{D}(k),
\quad |z| \le 1,
\label{add-160801-01}
\end{equation}
where $\wh{\vc{\pi}}(z) = \sum_{k=0}^{\infty}z^k\vc{\pi}(k)$.
Postmultiplying both sides of (\ref{add-160801-01}) by $\vc{e}$ and
rearranging the terms of the resulting equation, we have
\begin{eqnarray}
\mu (1-z){\rmd \over \rmd z} \wh{\vc{\pi}}(z)\vc{e}
&=& - \wh{\vc{\pi}}(z)\sum_{k=0}^{\infty}z^k\vc{D}(k)\vc{e}
\nonumber
\\
&=& \wh{\vc{\pi}}(z)\sum_{k=1}^{\infty}(1-z^k) \vc{D}(k)\vc{e},~~~|z| \le 1,
\qquad
\label{nece-dif}
\end{eqnarray}
where we use $\sum_{k=0}^{\infty}\vc{D}(k)\vc{e} = \vc{0}$ in the
second equality.  Furthermore, it follows from (\ref{nece-dif}) that
\[
\mu {\rmd \over \rmd z}\wh{\vc{\pi}}(z)\vc{e}
= \wh{\vc{\pi}}(z) \sum_{k=1}^{\infty} {1-z^k \over 1-z} \vc{D}(k)\vc{e}.
\]
Integrating both sides of this equation over $z\in (0,1)$ and using
$\wh{\vc{\pi}}(z) \ge \wh{\vc{\pi}}(0) = \vc{\pi}(0)$, we have
\begin{eqnarray}
\mu \{\wh{\vc{\pi}}(1) - \vc{\pi}(0)\}\vc{e}
&=& 
\sum_{k=1}^{\infty} 
\int_0^1 {1-z^k \over 1-z} \wh{\vc{\pi}}(z) \rmd z \cdot \vc{D}(k)\vc{e}
\nonumber
\\
&\ge& \vc{\pi}(0)  \sum_{k=1}^{\infty}\vc{D}(k)\vc{e} 
\int_0^1 {1-z^k \over 1-z} \rmd z.
\label{nece-ineq-0}
\end{eqnarray}
Note here that
\begin{eqnarray*}
\wh{\vc{\pi}}(1)\vc{e} 
&=& 1,
\\
\int_0^1{1-z^k \over 1-z}\rmd z
&=& \sum_{\ell=1}^k{1\over \ell}
\ge \log(k+1)
\nonumber
\\
&\ge& \log(k+\rme){\log 2 \over \log (1 + \rme)},
\quad k \in \bbN.
\end{eqnarray*}
Substituting these into (\ref{nece-ineq-0}), we obtain
\begin{eqnarray}
\lefteqn{
\vc{\pi}(0) 
\sum_{k=1}^{\infty} \log(k+\rme) \vc{D}(k)\vc{e}
}
\quad &&
\nonumber
\\
&\le& {\mu\log (1 + \rme) \over \log2} \{1 - \vc{\pi}(0)\vc{e} \}.
\label{add-160801-02}
\end{eqnarray}
Since $\vc{\pi}(0) > \vc{0}$ and $0 < \vc{\pi}(0)\vc{e} < 1$ (see
Remark~\ref{rem-Q}), the inequality (\ref{add-160801-02}) completes
the proof.
\end{proof}

\section{Extension to multiclass case}
\label{sec-Multi}

In this section, we consider an infinite-server queue with a
multiclass batch Markovian arrival process (MBMAP) and class-dependent
exponential service times. We assume that arriving customers are
classified into $K$ classes and the set of class indices is denoted by
$\bbK := \{1,2,\dots,K\}$. For each $\nu\in\bbK$, the service times of class $\nu$ customers are i.i.d.~with the exponential distribution having mean
$1/\mu_{\nu}\in (0,\infty)$.

The MBMAP is an extension of the BMAP described in
Section~\ref{sec-Model}. As in Section~\ref{sec-Model}, the MBMAP has
the background continuous-time Markov chain $\{J(t);t\geq 0\}$
with state space $\bbD$ and irreducible infinitesimal generator
$\vc{D}$.  For $\nu \in \bbK$, let $N_{\nu}(t)$, $t \ge 0$, denote
the total number of class $\nu$ customers who arrive from the MBMAP during the time
interval $(0,t]$, where $N_{\nu}(0)=0$.  Let
$N(t) = \sum_{\nu\in\bbK} N_{\nu}(t)$ for $t \ge 0$. We then assume
that, for $i,j\in\bbD$,
\begin{eqnarray*}
\lefteqn{
\PP(N(t+\Delta t) - N(t)=0,J(t+\Delta t)=j\ |\ J(t)=i)
}\quad&&
\\
&=&\left\{
\begin{array}{ll}
1+D_{i,i}(0)\Delta t + o(\Delta t),	& i=j\in\bbD,
\\
D_{i,j}(0)\Delta t + o(\Delta t),	& \mbox{otherwise},
\end{array}
\right.
\end{eqnarray*}
where $\vc{D}(0):=(D_{i,j}(0))_{i,j\in\bbD}$ is a diagonally dominant matrix with negative diagonal and nonnegative off-diagonal elements.
We also assume that, for $\nu \in \bbK$, $k \in \bbN$ and $i,j\in\bbD$,
\begin{eqnarray}
&&
\PP(N_{\nu}(t+\Delta t)-N_{\nu}(t)=k,J(t+\Delta t)=j \mid J(t)=i)
\nonumber
\\
&& \qquad {} = 
D_{\nu,i,j}(k)\Delta t + o(\Delta t),
\label{cond-D_{nu}(k)}
\end{eqnarray}
where $\vc{D}_{\nu}(k):=(D_{\nu,i,j}(k))_{i,j\in\bbD}$, $\nu\in\bbK$,
$k\in\bbN$, is a nonnegative matrix such that $\vc{D}(0) +
\sum_{\nu\in\bbK}\sum_{k=1}^{\infty} \vc{D}_{\nu}(k)$ is equal to the
infinitesimal generator of the background Markov chain $\{J(t)\}$,
i.e.,
\begin{equation}
\vc{D}(0)
+ \sum_{\nu\in\bbK}\sum_{k=1}^{\infty} \vc{D}_{\nu}(k) = \vc{D}.
\label{eqn-sum-D_{nu}(k)}
\end{equation}
It follows from (\ref{cond-D_{nu}(k)}) and (\ref{eqn-sum-D_{nu}(k)})
that the classes of the customers in a batch are same and thus their
service times are independently distributed with the same exponential distribution.

To avoid triviality, we assume that
\[
\sum_{k=1}^{\infty} \vc{D}_{\nu}(k)\vc{e}\neq\vc{0}\quad \mbox{for all $\nu \in \bbK$}.
\]
As a result, the MBMAP is characterized by $\{\vc{D}(0),
\vc{D}_{\nu}(k);\nu \in \bbK, k\in \bbN\}$.  In what follows, we
denote the MBMAP described above by MBMAP $\{\vc{D}(0),
\vc{D}_{\nu}(k);\nu \in \bbK, k\in \bbN\}$. In addition, we denote the
multiclass infinite-server queue described above by
MBMAP${}_K$/M${}_K$/$\infty$, where the subscript ``$K$" represents
the number of classes.

Let $\vc{L}(t) = (L_1(t),L_2(t),\dots,L_K(t))$ for $t \ge 0$, where
$L_{\nu}(t)$ denotes the number of class $\nu$ customers in the system
at time $t$. It then follows that the joint stochastic process
$\{(\vc{L}(t),J(t)); t \ge 0\}$ is an irreducible Markov chain with
state space $\bbZ_+^K \times \bbD$.
\begin{thm}\label{thm-multiclass}
The Markov chain $\{(\vc{L}(t),J(t))\}$ is ergodic if and only if
there exists some finite constant $C > 0$ such that
\begin{eqnarray}
\sum_{k=1}^{\infty}\log(k+\rme) \vc{D}_{\ast}(k)\vc{e} 
\le C\vc{e},
\label{stability-cond-multiclass}
\end{eqnarray}
where $\vc{D}_{\ast}(k) = \sum_{\nu \in \bbK} \vc{D}_{\nu}(k)$ for $k
\in \bbN$.
\end{thm}

\begin{rem}
Theorem~\ref{thm-multiclass} is a generalization of
\cite[Lemma~2]{Cong94}, which presents a necessary and sufficient
condition for the stability of a multiclass infinite-server queue with
batch Poisson arrivals and class-dependent exponential service times.
\end{rem}

{\sc Proof of Theorem~\ref{thm-multiclass}.~\,} Besides the original
MBMAP${}_K$/M${}_K$/$\infty$ queue, we consider two
MBMAP${}_K$/M${}_K$/$\infty$ queues, denoted by Queues 1 and 2,which
are fed by the same arrival process as that of the original queue,
i.e., fed by MBMAP $\{\vc{D}(0), \vc{D}_{\nu}(k);\nu \in \bbK, k\in
\bbN\}$. In Queue 1 (resp.~2), all the service times are i.i.d.~with
an exponential distribution having mean $1/\mu_{\min}$
(resp.~$1/\mu_{\max}$), where
\[
\mu_{\min} = \min_{\nu\in\bbK}\mu_{\nu},
\qquad
\mu_{\max} =
\min_{\nu\in\bbK}\mu_{\nu}.
\]
Clearly, Queues 1 and 2 can be considered single-class BMAP/M/$\infty$
queues when the class of customers are ignored, where the arrival
process is reduced to BMAP $\{\vc{D}(0),\ \vc{D}_{\ast}(k); k \in
\bbN\}$.

Let $|\vc{L}(t)| = \sum_{\nu\in\bbK}L_{\nu}(t)$ for $t \ge 0$, which
denotes the total number of customers in the system of the original
MBMAP${}_K$/M${}_K$/$\infty$ queue at time $t$.  For $i=1,2$, let
$L^{(i)}(t)$, $t \ge 0$, denote the total number of customers in the
system of Queue $i$ at time $t$. From the assumption of Queues 1 and
2, we can construct the three joint processes $\{(\vc{L}(t),J(t));t
\ge 0\}$, $\{(L^{(1)}(t),J(t));t \ge 0\}$, $\{(L^{(2)}(t),J(t));t \ge
0\}$ in a common probability space such that the following pathwise
ordered relation holds:
\begin{equation}
L^{(2)}(t) \le |\vc{L}(t)| \le L^{(1)}(t)\quad \mbox{for all $t \ge 0$},
\label{ordered-property}
\end{equation}
which is proved in Appendix~\ref{proof-pathwise-ordered}.

It should be noted that $\{(L^{(1)}(t),J(t))\}$ and
$\{(L^{(2)}(t),J(t))\}$ are Markov chains of the same type as
$\{(L(t),J(t))\}$ discussed in the previous section. It thus follows
from Theorem~\ref{cond-batch} that (\ref{stability-cond-multiclass})
holds if and only if $\{(L^{(1)}(t),J(t))\}$ and
$\{(L^{(2)}(t),J(t))\}$ are ergodic.

We now suppose that $\{(L^{(1)}(t),J(t))\}$ is ergodic. It then
follows from (\ref{ordered-property}) that $\{L^{(1)}(t)\}$ and thus
$\{|\vc{L}(t)|\}$ take the value of zero infinitely many times w.p.1
and the mean recurrence time to state 0 is finite (see, e.g.,
\cite[Chapter 8, Definitions~5.1, 5.2 and 5.4]{Brem99}). Therefore,
$\{(\vc{L}(t),J(t))\}$ is ergodic.

On the other hand, we suppose that $\{(L^{(2)}(t),J(t))\}$ is not
ergodic, i.e., is transient or null-recurrent. Note that if
$\{(L^{(2)}(t),J(t))\}$ is transient then $\{L^{(2)}(t)\}$ and thus
$\{|\vc{L}(t)|\}$ take the value of zero, at most, finitely many times
with some positive probability. Note also that if
$\{(L^{(2)}(t),J(t))\}$ is null-recurrent then the mean recurrence
times to state 0 of $\{L^{(2)}(t)\}$ and thus $\{|\vc{L}(t)|\}$ are
infinite. Therefore, in both cases, $\{(\vc{L}(t),J(t)); t \ge 0\}$ is
not ergodic.

As a result, the above argument shows that
(\ref{stability-cond-multiclass}) holds if and only if
$\{(\vc{L}(t),J(t))\}$ is ergodic.  \qed

\section{Conclusions}
\label{Contribution}

In this paper, we have shown that the BMAP/M/$\infty$ queue is stable
if and only if the logarithms of the sizes of arriving batches have a
finite mean. We also have extended this result to an infinite-server
queue with the MBMAP and class-dependent exponential service times.

We expect that the stability condition of this paper holds for a more
general infinite-server queue with the MBMAP and class-dependent and
phase-type service times, which would be proved in the same way as the
model considered in this paper.  It should be noted that the set of
phase-type distributions is dense in the set of distribution on
$[0,\infty)$ (see \cite{Asmu93}). Thus, we can also conjecture that
  the stability condition of this paper is extended to an
  infinite-server queue with the MBMAP and class-dependent and
  light-tailed service times. This problem is challenging because the
  joint queue length process is not necessarily Markovian unlike the
  queues considered in this paper.


\bibliographystyle{plain}
\bibliography{hm-160809}  

\appendix

\section{Proof of the pathwise ordered relation}\label{proof-pathwise-ordered}

Let $T_n$, $n \in \bbN$, denote the $n$th arrival time of batches from
MBMAP $\{\vc{D}(0), \vc{D}_{\nu}(k);\nu \in \bbK, k\in \bbN\}$, where
\[
0 < T_1 < T_2 < \cdots.
\]
Let $c_n$ and $B_n$, $n \in \bbN$, denote the class and batch size,
respectively, of the batch arriving at time $T_n$. Furthermore, let
$\{U_m;m\in\bbN\}$ denote a sequence of i.i.d.\ random variables
with a uniform distribution on the interval $(0,1)$. We then
define $S_m$, $\ol{S}_m$ and $\ul{S}_m$, $m \in \bbN$, as random
variables such that, for $A_{n-1}+1 \le m \le A_n$ and $n \in \bbN$,
\begin{eqnarray}
S_m 
&=& -{1 \over \mu_{c_n}}\log U_m,
\label{defn-S_m}
\\
\ol{S}_m 
&=& -{1 \over \mu_{\min}}\log U_m,
\label{defn-ol{S}_m}
\\
\ul{S}_m 
&=& -{1 \over \mu_{\max}}\log U_m,
\label{defn-ul{S}_m}
\end{eqnarray}
where $A_0 = 0$ and $A_n = \sum_{k=1}^n B_k$ for $n \in \bbN$. It
follows from (\ref{defn-S_m})--(\ref{defn-ul{S}_m}) that
\begin{align}
\PP(S_m \le x)
&= 1 - \exp\{-\mu_{c_n} x\}, & x & \ge 0,
\label{dist-S_m}
\\
\PP(\ol{S}_m \le x) 
&= 1 - \exp\{-\mu_{\min} x\}, & x & \ge 0,
\\
\PP(\ul{S}_m \le x) 
&= 1 - \exp\{-\mu_{\max} x\}, & x & \ge 0.
\label{dist-ul{S}_m}
\end{align}
In addition, since $\mu_{\min} \le \mu_{c_n} \le \mu_{\max}$, we have
\begin{equation}
\ul{S}_m \le S_m \le \ol{S}_m,\qquad m \in \bbN.
\label{ineqn-S_m}
\end{equation}

Based on (\ref{dist-S_m})--(\ref{dist-ul{S}_m}), we assume that
$\{S_m;A_{n-1}+1 \le m \le A_n\}$, $\{\ol{S}_m;A_{n-1}+1 \le m \le
A_n\}$ and $\{\ul{S}_m;A_{n-1}+1 \le m \le A_n\}$ are the service
times of the customers in the $n$th batch arriving at the original
MBMAP${}_K$/M${}_K$/$\infty$ queue, Queues 1 and 2, respectively. We
then fix $|\vc{L}(t)|$, $L^{(1)}(t)$ and $L^{(2)}(t)$, $t \ge 0$ such
that
\begin{eqnarray}
|\vc{L}(t)|
&=& \sum_{n=1}^{\infty}\sum_{m=A_{n-1}+1}^{A_n} 
I(T_n \le t < T_n+S_m),
\label{defn-L(t)}
\\
L^{(1)}(t)
&=& \sum_{n=1}^{\infty}\sum_{m=A_{n-1}+1}^{A_n}
I(T_n \le t < T_n+\ol{S}_m),
\label{defn-L^{(1)}(t)}
\\
L^{(2)}(t)
&=& \sum_{n=1}^{\infty}\sum_{m=A_{n-1}+1}^{A_n}
I(T_n \le t < T_n+\ul{S}_m),
\label{defn-L^{(2)}(t)}
\end{eqnarray}
where $I(\chi)$ denotes the indicator function of any event $\chi$. It
is easy to see that $\{|\vc{L}(t)|\}$, $\{L^{(1)}(t)\}$ and
$\{L^{(2)}(t)\}$ can be considered the total queue length processes of
the original MBMAP${}_K$/M${}_K$/$\infty$ queue, Queues 1 and 2,
respectively, which are fed by the common MBMAP. Furthermore, combining
(\ref{ineqn-S_m}) with (\ref{defn-L(t)})--(\ref{defn-L^{(2)}(t)}), we
obtain the pathwise ordered relation (\ref{ordered-property})
between $\{|\vc{L}(t)|\}$, $\{L^{(1)}(t)\}$ and $\{L^{(2)}(t)\}$.

\end{document}